\documentclass[a4paper,reqno]{amsart}
\usepackage{graphicx}
\usepackage{amsmath,amssymb,amsthm}
\usepackage{mathrsfs,tikz,caption,subcaption}
\usepackage{hyperref}
\usepackage{verbatim}

\theoremstyle{definition}
\newtheorem{thm}{Theorem}[section]

\newtheorem{lemm}[thm]{Lemma}

\DeclareMathOperator{\Mod}{Mod }

\DeclareMathAlphabet{\mathpzc}{OT1}{pzc}{m}{it}

\begin{document}
\title[Quasiconformal parametrization with small dilatation]{Quasiconformal parametrization of metric surfaces with small dilatation}

\author{Matthew Romney}
\address{Department of Mathematics, University of Illinois at Urbana-Champaign, Urbana, IL 61801, USA.} 
\email{romney2@illinois.edu}


\maketitle

\begin{abstract}
We verify a conjecture of Rajala: if $(X,d)$ is a metric surface of locally finite Hausdorff 2-measure admitting some (geometrically) quasiconformal parametrization by a simply connected domain $\Omega \subset \mathbb{R}^2$, then there exists a quasiconformal mapping $f: X \rightarrow \Omega$ satisfying the modulus inequality $2\pi^{-1}\Mod \Gamma \leq \Mod f\Gamma \leq 4\pi^{-1}\Mod \Gamma$ for all curve families $\Gamma$ in $X$. This inequality is the best possible. Our proof is based on an inequality for the area of a planar convex body under a linear transformation which attains its Banach-Mazur distance to the Euclidean unit ball.
\end{abstract}

\section{Introduction}

A growing body of recent literature has studied the following {\it quasiconformal uniformization problem}: for a metric space $(X,d)$ homeomorphic to a domain $\Omega$ in $\mathbb{R}^n$ or $\mathbb{S}^n$, under what conditions does there exist a quasiconformal (or quasisymmetric) parametrization of $X$ by $\Omega$? This question originated largely in the work of Semmes; see especially \cite{Sem:96b} and \cite[Qu. 3-7]{HeiS:97}. A landmark paper of Bonk and Kleiner \cite{BonkKleiner} gives a complete description of those spaces $X$ admitting a quasisymmetric parametrization by $\mathbb{S}^2$ under the assumption that $X$ is Ahlfors 2-regular; a necessary and sufficient condition is that $X$ be linearly locally contractible. 

A similar theorem for geometrically quasiconformal parametrizations was recently proven by Rajala \cite{Raj16} in the setting of metric surfaces homeomorphic to $\mathbb{R}^2$ or $\mathbb{S}^2$ with locally finite Hausdorff 2-measure. A new condition called {\it reciprocality} is introduced which is necessary and sufficient for the existence of the desired quasiconformal parametrization. The original Bonk--Kleiner theorem can then be obtained as a corollary. We refer the reader to the introduction of Rajala's paper for additional background and references. See also Lytchak and Wenger \cite{LyWen:17} for other new results on quasiconformal parametrizations in somewhat the same spirit. 

We recall now the relevant definitions. Let $(X,d,\mu)$ be a metric measure space. Given a family $\Gamma$ of curves in $X$, the {\it $p$-modulus} of $\Gamma$ is 
$$\Mod_p \Gamma := \inf_\rho \int_X \rho^p d\mu ,$$
the infimum taken over all Borel functions $\rho:X \rightarrow [0, \infty]$ such that $\int_{\gamma} \rho\, ds \geq 1$ for every locally rectifiable curve $\gamma \in \Gamma$. A homeomorphism $f: (X,d,\mu) \rightarrow (Y, d', \nu)$ is {\it $K$-geometrically quasiconformal with exponent $p$} if $$K^{-1}\Mod_p \Gamma \leq \Mod_p f(\Gamma) \leq K\Mod_p \Gamma$$ for all curve families $\Gamma$ in $X$. The smallest value $K_O$ such that $\Mod_p \Gamma \leq K_O\Mod_p f(\Gamma)$ for all curve families $\Gamma$ in $X$ is called the {\it outer dilatation} of $f$. Similarly, the smallest value $K_I$ such that $\Mod_p f(\Gamma) \leq K_I \Mod_p \Gamma$ for all curve families $\Gamma$ in $X$ is the {\it inner dilatation}. 

If $p$ is understood, we say simply that $f$ is {\it $K$-quasiconformal} or {\it quasiconformal}. In this note, we always take $p=2$ and we write $\Mod \Gamma$ in place of $\Mod_2 \Gamma$. We will assume that a metric space $(X,d)$ is equipped with the Hausdorff 2-measure.

The same paper of Rajala also examines a related question: if such a quasiconformal parametrization exists, can one find a quasiconformal mapping which improves the dilatation constants $K_O$ and $K_I$ to within some universal constants? If so, what is the best result of this type? Rajala obtains the following theorem \cite[Thm. 1.5]{Raj16}:

\begin{thm}{(Rajala)} \label{thm:rajala}
Let $\Omega \subset \mathbb{R}^2$ be a simply connected domain and $(X,d)$ a metric space of locally finite Hausdorff 2-measure. There exists a quasiconformal homeomorphism $f: X \rightarrow \Omega$ if and only if there exists a 2-quasiconformal homeomorphism $f: X \rightarrow \Omega$. 
\end{thm}

This result is proved using the measurable Riemann mapping theorem along with the classical John's theorem on convex bodies. The latter theorem asserts, in part, that any convex body $A$ in $\mathbb{R}^n$ contains a unique ellipsoid $E$ of maximal volume satisfying $E \subset A \subset \sqrt{n}E$, where the constant $\sqrt{n}$ is the best possible. The constant 2 in Theorem \ref{thm:rajala} is derived from the constant $\sqrt{2}$ in John's theorem for dimension two.

In this note, we prove the following improvement to Theorem \ref{thm:rajala}, which was conjectured by Rajala in \cite{Raj16}. 

\begin{thm}\label{thm:main}
Let $\Omega \subset \mathbb{R}^2$ be a simply connected domain and $(X,d)$ a metric space of locally finite Hausdorff 2-measure. There exists a quasiconformal homeomorphism $f: X \rightarrow \Omega$ if and only if there exists a quasiconformal homeomorphism $f: X \rightarrow \Omega$ satisfying 
\begin{equation} \label{equ:modulus}
\frac{2}{\pi}\Mod \Gamma \leq \Mod f\Gamma \leq \frac{4}{\pi}\Mod \Gamma.
\end{equation} 
\end{thm}

Rajala's techniques, together with standard volume ratio estimates (see for instance \cite[Thm. 6.2]{Ball:97}), guarantee the existence of a quasiconformal map $f_O: X \rightarrow \Omega$ with outer dilatation $K_O \leq \pi/2$, and a quasiconformal map $f_I: X \rightarrow \Omega$ with inner dilatation $K_I \leq 4/\pi$. The improvement in Theorem \ref{thm:main} is in finding a map which satisfies both modulus inequalities simultaneously. 

Inequality \eqref{equ:modulus} cannot be improved, as shown by taking $(X,d) = (\mathbb{R}^2,{\|\cdot\|_\infty)}$, where $\|\cdot\|_\infty$ is the $\ell^\infty$ metric. That is, every quasiconformal map $f: (\mathbb{R}^2,{\|\cdot\|_\infty)} \rightarrow \mathbb{R}^2$ must satisfy $K_O \geq \pi/2$ and $K_I \geq 4/\pi$. Moreover, the identity map $\iota: (\mathbb{R}^2,{\|\cdot\|_\infty)} \rightarrow \mathbb{R}^2$ satisfies \eqref{equ:modulus}. These facts are proved in Example 2.2 of \cite{Raj16}.

The simple connectedness assumption is essential to Theorem \ref{thm:main}. For example, any $K$-quasiconformal mapping $f$ between the annular regions $\{x \in \mathbb{R}^2: 1 < |x| < a\}$ and $\{x \in \mathbb{R}^2: 1 < |x| < b\}$, $b \geq a$, must satisfy $K \geq \log b/\log a$. In particular, for any $K' \geq 1$ there exist annuli $A_1, A_2 \subset \mathbb{R}^2$ such that any $K$-quasiconformal map $f: A_1 \rightarrow A_2$ must satisfy $K \geq K'$. A similar fact holds for wedge domains in $\mathbb{R}^n$, $n \geq 3$, which is one indication that a result like Theorem \ref{thm:main} is only possible in dimension two. See V\"ais\"al\"a \cite[Sec. 39-40]{Vais1} for a discussion of quasiconformal mappings between annular and wedge domains. 

Finally, Theorem \ref{thm:main} remains true when $\mathbb{R}^2$ is replaced by $\mathbb{S}^2$, though for simplicity we do not address that case explicitly.  

\section{An area inequality for planar convex bodies} 

The key innovation for proving Theorem \ref{thm:main} is to replace the application of John's theorem by the two lemmas in this section, after which Theorem \ref{thm:main} follows by a straightforward modification of Rajala's proof. See the introductory notes by Ball \cite{Ball:97} for an overview of John's theorem and related results on volume ratios. However, we have not found any result comparable to our lemmas in the literature on convex bodies. 

In the following, $|E|$ will denote the area of the set $E \subset \mathbb{R}^2$. We also let $L(E) = \sup\{|z|: z \in E\}$ denote the outer radius of $E$ and $\ell(E) = \inf\{|z|: z \notin E\}$ denote the inner radius of $E$. A {\it convex body} is a compact convex set $A$ in $\mathbb{R}^2$ with nonempty interior; it is {\it symmetric} if $z \in A$ implies $-z \in A$. There is a natural correspondence between the set of norms on $\mathbb{R}^2$ and the set of symmetric convex bodies in $\mathbb{R}^2$. Namely, the unit ball for a norm on $\mathbb{R}^2$ is a symmetric convex body, while for any symmetric convex body $A$ the function $p(x) := \inf\{t>0: x/t \in A\}$ defines a norm on $\mathbb{R}^2$. Terms such as {\it ellipse} and {\it polygon} should be understood as including the interior of the respective objects.  

For a convex body $A \subset \mathbb{R}^2$ and a linear transformation $T \in GL(2,\mathbb{R})$, let $r(A,T) = L(TA)/\ell(TA)$. Define $\rho(A) = \inf r(A,T),$ the infimum taken over all $T \in GL(2,\mathbb{R})$. Notice that $\rho(A) \leq \sqrt{2}$ by John's theorem. Expressed in different terms, $\rho(A)$ is the (multiplicative) Banach-Mazur distance between $A$ and the closed Euclidean unit ball in $\mathbb{R}^2$. 

It is easy to verify that there is a matrix $T \in GL(2, \mathbb{R})$ such that $r(A,T)$ attains $\rho(A)$. Consider the family $\mathcal{T} = \{T \in GL(2, \mathbb{R}): 1/\sqrt{2} \leq \ell(TA) \leq L(TA) \leq 1\}$. By John's theorem, restricting to $T \in \mathcal{T}$ does not affect the infimal value of $r(A,T)$. For such a map
$$T = \begin{pmatrix} a & b \\ c & d \end{pmatrix},$$
we must have $a^2 + c^2 \leq \ell(A)^{-2}$ and $b^2 + d^2 \leq \ell(A)^{-2}$. This is seen by looking at the action of $T$ on the test points $\ell(A)e_1$ and $\ell(A)e_2$. We also have $|\det A| =|ad-bc| \geq L(A)^{-2}/2$. Hence the set $\mathcal{T}$ is compact as a subset of $GL(2,\mathbb{R})$ and it follows that a nonzero map $T$ minimizing $r(A,T)$ must exist.

\begin{lemm} \label{lemm:convex_body}
Let $A \subset \mathbb{R}^2$ be a symmetric convex body and $T \in GL(2,\mathbb{R})$ a linear map such that $r(A,T) = \rho(A)$. Then the image of $A$ satisfies $2L(TA)^2 \leq |TA| \leq 4\ell(TA)^2$.   
\end{lemm}
\begin{proof}
Let $\widetilde{A} = TA$. Without loss of generality we can assume that the outer radius satisfies $L(\widetilde{A}) = 1$. Let $\ell = \ell(\widetilde{A})$. Then from John's theorem it follows that $2^{-1/2} \leq \ell \leq 1$.  Where convenient, we will use complex notation for points in $\mathbb{R}^2$. For $\theta \in [0, 2\pi)$, let $z_\theta$ denote the unique point in $\partial \widetilde{A} \cap \{e^{i\theta}r:r>0\}$. By rotating if necessary, we will assume that $|z_0| = 1$. 

We first need a fact about the existence of contact points with the circles $S(0,\ell)$ and $S(0,1)$. Specifically, we claim that there exist values $0 = \theta_0 < \theta_1 < \theta_2 < \theta_3< \pi$ such that $|z_{\theta_0}| = |z_{\theta_2}| = 1$ and $|z_{\theta_1}| = |z_{\theta_3}| = \ell$. Suppose this does not hold. Then there exist $0<\theta_1 \leq \theta_3 < \pi$ such that $|z_{\theta_1}| = |z_{\theta_3}| = \ell$, $|z_\theta| < 1$ whenever $\theta_1 < \theta < \theta_3$, and $|z_\theta|> \ell$ whenever $0 < \theta<\theta_1$ or $\theta_3 < \theta < \pi$. Observe that if $\theta, \theta'$ are such that $|z_{\theta}| = 1$ and $|z_{\theta'}| = \ell$, then $|\theta - \theta'| \geq \cos^{-1}(\ell)$. In particular, $\theta_1 \geq \cos^{-1}(\ell)$ and $\theta_3 \leq \pi - \cos^{-1}(\ell)$. 

Consider now a small linear stretch in the direction $(\theta_1 + \theta_3)/2$. Expressed in a suitable orthonormal basis $\{v_1, v_2\}$ for $\mathbb{R}^2$, this linear stretch takes the form $T_\lambda: (x,y) \mapsto (\lambda x, y)$ for some sufficiently small parameter $\lambda > 1$. 

Let $\theta = (\theta_3-\theta_1)/2$. 
For sufficiently small $\epsilon>0$, consider the function
$$R_\epsilon(\lambda) = \frac{\lambda^2\cos^2(\theta + \cos^{-1}(\ell)-\epsilon) +  \sin^2(\theta + \cos^{-1}(\ell)-\epsilon)}{\lambda^2\ell^2\cos^2(\theta+\epsilon) +  \ell^2 \sin^2(\theta+\epsilon)}.$$
We obtain the functions $R_\epsilon(\lambda)$ by considering the (Euclidean) norm of the image of the points $e^{i(\theta_1 -\cos^{-1}(\ell)+\epsilon)}$ (the numerator) and $\ell^2 e^{i(\theta_1-\epsilon)}$ (the denominator), as expressed relative to the basis $\{v_1, v_2\}$. Then $R_\epsilon(\lambda)$ is an upper bound for $r(A,T_\lambda T)$ for sufficiently small $\lambda$ and satisfies $R_\epsilon(1) = r(A,T)$. In particular, $\frac{d}{d\lambda}r(A,T_{\lambda} T) \leq R_0'(\lambda)$. We compute
$$R_0'(\lambda) = \frac{-2\lambda \sqrt{1-\ell^2} \sin(2 \theta +\cos^{-1}(\ell))}{\ell^2 \left(\sin^2\theta + \lambda^2 \cos^2 \theta_1 \right)^2}.$$
Since $2\theta_1 + \cos^{-1}(\ell)< \pi$, we see that $R'(1)<0$. This contradicts the minimality of $r(A,T)$. The existence of the desired values $0 = \theta_0 < \theta_1 < \theta_2 < \theta_3 < \pi$ now follows.

We now estimate $|\widetilde{A}|$ from above. Write $\theta_\ell = \cos^{-1}(\ell)$. By covering $\widetilde{A}$ with the triangles $[0,e^{i(\theta_1 - \theta_\ell)},e^{i(\theta_1 + \theta_\ell)}]$, $[0,1,e^{i(\theta_3 - \theta_\ell)}, e^{i(\theta_3 + \theta_\ell)}]$ and the set $$\{re^{i\theta}: 0 \leq r \leq 1, \theta \in [0, \theta_1 - \theta_\ell] \cup [\theta_1 + \theta_\ell, \theta_2 - \theta_\ell] \cup [\theta_2 + \theta_\ell, \pi] \},$$ along with their reflections about the origin, we obtain 
$$|\widetilde{A}| \leq M(\ell) := \pi-4\cos^{-1}(\ell) + 4\ell \sqrt{1- \ell^2}.$$
See Figure \ref{fig:shaded_area_outer}. Observe that $M(2^{-1/2}) = 2 = 4(2^{-1/2})^2$, so the right inequality $|\widetilde{A}| \leq 4\ell^2$ holds for $\ell = 2^{-1/2}$. Next, compute $M'(\ell) = 8\sqrt{1-\ell^2}$. Since this satisfies $M'(\ell) \leq 8\ell$ when $2^{-1/2} \leq \ell \leq 1$, we obtain  $|\widetilde{A}| \leq 4\ell^2$ holds for all $\ell \in [2^{-1/2}, 1]$. 

We can estimate $|\widetilde{A}|$ from below using the polygons $[0,1,\ell e^{i\theta_\ell}]$,  $[0,-1,\ell e^{i(\pi - \theta_\ell)}]$, $[0,\ell e^{i \theta_2-\theta_\ell}, e^{i\theta_2}, \ell e^{i \theta_2+\theta_\ell} ]$ and the set $$\{re^{i\theta}: 0 \leq r \leq \ell, \theta \in [\theta_\ell, \theta_2 - \theta_\ell] \cup [\theta_2 + \theta_\ell, \pi - \theta_\ell] \}.$$ See Figure \ref{fig:shaded_area_inner}. This gives
$$|\widetilde{A}| \geq m(\ell) := (\pi - 4\cos^{-1}(\ell)) \ell^2 + 4 \ell\sqrt{1-\ell^2}.$$
Now $m(2^{-1/2}) = 2$, so the left inequality $2L^2 \leq |\widetilde{A}|$ holds when $\ell = 2^{-1/2}$. Since $m'(\ell) = 2\pi \ell + 4\sqrt{1-\ell^2} - 8\ell \cos^{-1}(\ell) \geq 0$, we obtain $2 \leq |\widetilde{A}|$ for all $\ell \in [2^{-1/2}, 1]$. This completes the proof.
\end{proof}

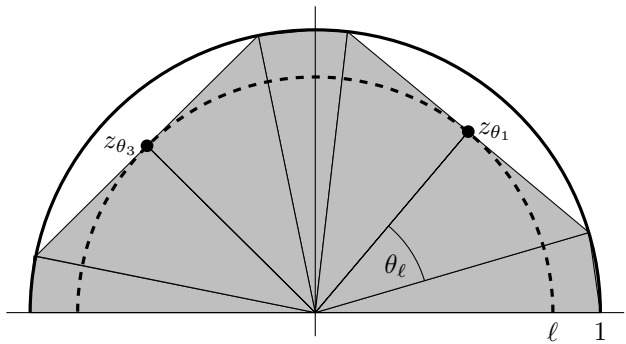
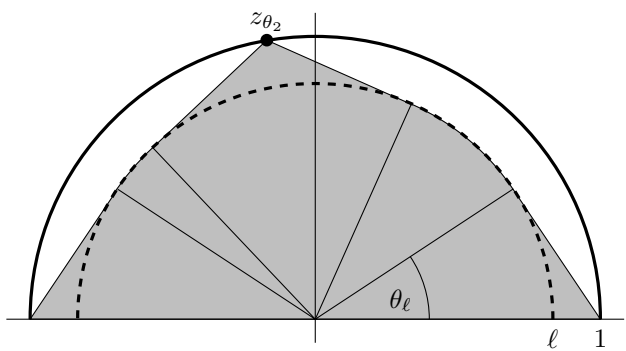
\begin{figure}[tb] \centering
\begin{subfigure}[tb]{\textwidth}
\centering
\begin{tikzpicture}[>=latex,scale=1.25]
	 \draw[fill=lightgray] (3,0) arc (0:16.4:3);
	 \draw[fill=lightgray] (0,0) -- (3,0) --(2.88,.85) -- (0,0);
	 \draw[fill=lightgray] (0,0) -- (.34,2.98) -- (0,3) -- (-.6,2.94) -- (0,0);	 
	 \draw[fill=lightgray] (0,0) -- (-3,0) --(-2.96,.6) -- (0,0); 
	  \draw[fill=lightgray] (0,0) -- (2.88,0.85) -- (1.61,1.92) node[right] {$z_{\theta_1}$} -- (0,0) -- (1.61,1.92) -- (.34,2.98) -- (0,0);
	  \draw[fill=lightgray] (0,0) -- (-.6,2.94) -- (-1.77,1.77) node[left] {$z_{\theta_3}$} -- (0,0) -- (-1.77,1.77) -- (-2.94,.6) -- (0,0);
	  \draw[very thick, fill=black] (1.61,1.92) circle [radius=.05];
	  \draw[very thick, fill=black] (-1.77,1.77) circle [radius=.05];

  \draw (-3.25,0) -- (3.25,0) ;
    \draw (0,-0.25) -- (0,3.25) ;
  \draw[very thick] (3,0) node[below] {1} arc (0:180:3) ;
  \draw[very thick, dashed] (2.5,0) node[below] {$\ell$} arc (0:180:2.5);
   \draw (1.15,.336) arc (16.2:49.8:1.2); 
  \draw (.85,.5) node {$\theta_\ell$}; 
\end{tikzpicture}

\caption{The convex body $\widetilde{A}$ is contained in the shaded region and its reflection about the origin.}
\label{fig:shaded_area_outer}
\end{subfigure} 

\begin{subfigure}[tb]{\textwidth}
\centering

\begin{tikzpicture}[>=latex,scale=1.25]
	 \draw[dashed, fill=lightgray] (2.5,0) arc (0:180:2.5);
	  \draw[fill=lightgray] (0,0) -- (3,0) -- (2.08,1.38) -- (0,0);
	  \draw[fill=lightgray] (0,0) -- (-1.71,1.82)-- (-.51,2.96) node[above] {$z_{\theta_2}$} -- (1.01,2.28) -- (0,0);
	  \draw[fill=lightgray] (0,0) -- (-3,0) -- (-2.08,1.38)  -- (0,0);
	  \draw[fill=lightgray] (2.08,1.38) arc (33.6:65.8:2.5);
	  \draw[fill=lightgray] (-1.71,1.82) arc (133.2:147.4:2.7);
	  \draw[very thick, fill=black] (-.51,2.96) circle [radius=.05];
  \draw (-3.25,0) -- (3.25,0) ;
  \draw (0,-0.25) -- (0,3.25) ;
  \draw[very thick] (3,0) node[below] {1} arc (0:180:3) ;
  \draw[very thick, dashed] (2.5,0) node[below] {$\ell$} arc (0:180:2.5);
   \draw (1.2,0) arc (0:33.6:1.2); 
  \draw (.9,.2) node {$\theta_\ell$}; 
\end{tikzpicture}

\caption{The convex body $\widetilde{A}$ contains the shaded region and its reflection about the origin.} 
\label{fig:shaded_area_inner}
\end{subfigure} 

\caption{Estimating the area of $\widetilde{A}$.} 
\end{figure}

By taking $A = [-1,1] \times [-1,1]$ and $T$ to be the identity map, we see that Lemma \ref{lemm:convex_body} is sharp. 

The proof of the previous lemma allows us show the next fact, that the linear map which attains $\rho(A)$ is unique up to a conformal transformation. 

\begin{lemm} \label{lemm:uniqueness}
Let $A \subset \mathbb{R}^2$ be a symmetric convex body, and let $T_1, T_2 \in GL(2, \mathbb{R})$ be linear maps such that $r(A,T_1) = r(A,T_2) = \rho(A)$. Then $T_1 = \lambda QT_2$ for some $\lambda > 0$ and orthogonal transformation $Q \in O(2)$. 
\end{lemm}
\begin{proof}
Let $\widetilde{T}_1 = \lambda_1 Q_1T_1$ and $\widetilde{T}_2 = \lambda_2 T_2$, where we have chosen $\lambda_1, \lambda_2>0$ and $Q_1 \in O(2)$ so that $L(\widetilde{T}_1A) = L(\widetilde{T}_2A) = 1$ and that $|z_0| = 1$; here and throughout this proof $z_\theta$ denotes the unique point in $\partial \widetilde{T}_1A \cap \{e^{i\theta} r: r>0\}$. It suffices to prove the lemma with $\widetilde{T}_1, \widetilde{T}_2$ in place of $T_1, T_2$. 
Note that after making this reduction we must have $\lambda =1$.

Let $B$ denote the closed Euclidean unit ball in $\mathbb{R}^2$. The conclusion of the lemma holds if and only if $\widetilde{T_1}^{-1}(B) = \widetilde{T_2}^{-1}(B)$; thus it suffices to show that $\widetilde{T_1}\widetilde{T_2}^{-1}(\ell B) = \ell B$, where $\ell = 1/\rho(A) = \ell(\widetilde{T}_1A)$. The set $\widetilde{T_1}\widetilde{T_2}^{-1}(\ell B)$ is an ellipse $E'$, whose boundary consists of those points satisfying the equation $f(x,y) := ax^2 + bxy + cy^2 = 1$ for some $a,b,c \in \mathbb{R}$ with $a,c>0$ and $4ac-b^2>0$. Let $(x(\theta),y(\theta)) = (\ell \cos \theta, \ell \sin \theta)$ and let $f_0(\theta) = f(x(\theta),y(\theta))$. 

Recall the assumption that $r(A,\widetilde{T_1}) = r(A,\widetilde{T_2}) = \rho(A)$; this implies that $E' = \widetilde{T_1}\widetilde{T_2}^{-1}(\ell B) \subset \widetilde{T_1}A \subset \widetilde{T_1}\widetilde{T_2}^{-1}(B)$. In particular, if $\theta \in [0, \pi)$ is such that $|z_\theta| = 1$, then $\ell z_\theta \in E'$ by the second inclusion. On the other hand, if $\theta \in [0, \pi)$ is such that $|z_\theta| = \ell$, then $z_\theta \in \overline{\mathbb{R}^2 \setminus E'}$ by the first inclusion.  
It must follow that $f_0(0) = f_0(\pi) \leq 1$, $f_0(\theta_1) \geq 1$, $f_0(\theta_2) \leq 1$, and $f_0(\theta_3) \geq 1$, where $0 < \theta_1 < \theta_2 < \theta_3 < \pi$ satisfy $|z_{\theta_1}| = |z_{\theta_3}|= \ell$ and $|z_{\theta_2}| = 1$ as in the proof of Lemma \ref{lemm:convex_body}. 
Then $f_0$ must have at least three critical points on the interval $[0, \pi)$, unless in fact $f_0(\theta)$ is identically equal to $1$. This leads to a contradiction, since $f_0$ (being derived from the equation for an ellipse, or by inspection) cannot have more than two critical points over the interval $[0,\pi)$. The result follows. 
\end{proof}

One view on John's theorem is that it provides one with a canonical choice of ellipse associated to a convex body $A \subset \mathbb{R}^2$, namely the ellipse $E \subset A$ maximizing area. 
The point of the previous lemma is to justify a different notion of canonical ellipse for a convex body, related to minimizing distance to the Euclidean ball in the sense of the Banach-Mazur distance. 
This ellipse is obtained by taking $E = T^{-1}(B(0, \ell(TA)))$ for any $T \in GL(2, \mathbb{R})$ such that $r(A,T) = \rho(A)$. Lemma \ref{lemm:uniqueness} shows that this ellipse $E$ is independent of the choice of $T$.

\section{Proof of main theorem}

This section gives the proof of Theorem \ref{thm:main}. It is a modification of the proof which comprises Section 14 of \cite{Raj16}. As such, we will highlight the modifications while referring the reader to \cite{Raj16} for additional details. We will also follow the notation found there where convenient.

For a Lipschitz function $g: \Omega \subset \mathbb{R}^2 \rightarrow Z$ into a metric space $Z$, we use $MD(g,x)$ to denote the metric differential of Kirchheim \cite{Kirc:94} at the point $x \in \Omega$, which exists for a.e. $x \in \Omega$. As explained in \cite[Lem. 14.1, 14.2]{Raj16}, for every quasiconformal map $h: \Omega \rightarrow X$ there exist disjoint measurable sets $\Omega_j$ ($j=1,2,\ldots$) covering $\Omega$ up to a set of measure zero such that $h|\Omega_j$ is $j$-Lipschitz. The map $h|\Omega_j$ can be extended to a Lipschitz map $h_j: \mathbb{R}^2 \rightarrow \ell^\infty(X)$. Then for all $j \in \mathbb{N}$ and a.e $x \in \Omega$, $MD(h_j,x)$ is a non-zero norm on $\mathbb{R}^2$. This notation will be used in the following proof. 

\begin{proof}
Recall that we are assuming the existence of a quasiconformal homeomorphism $h = f^{-1}: \Omega \rightarrow (X,d)$. For a.e. $x \in \Omega$, we obtain a non-zero norm $G_x$ on $\mathbb{R}^2$ from the metric derivative of the function $h_j$ described above, where $j$ is such that $x \in \Omega_j$. For each such norm $G_x$, the set $C_x = \{y \in \mathbb{R}^2: G_x(y) \leq 1\}$ is a symmetric convex body in $\mathbb{R}^2$.

Let $T_x$ be an invertible linear mapping for which $L(T_xC_x)/\ell(T_xC_x) = \rho(C_x)$.  
Let $E_x = T_x^{-1}(B(0,\ell(T_xC_x)))$; this gives an ellipse field on $\Omega$ defined for a.e $x \in \Omega$. As we have seen from Lemma \ref{lemm:uniqueness}, the ellipse $E_x$ does not depend on our choice of $T_x$. Setting $E_x = B(x,1)$ for the remaining points in $\Omega$ gives an ellipse field defined on all $\Omega$. The associated complex dilatation is measurable and has a uniform bound less than 1. Observe that the analogous ellipse field in Rajala's proof was obtained using John's theorem instead.

Applying the measurable Riemann mapping theorem gives a quasiconformal mapping $\nu: \Omega \rightarrow \Omega$ such that 
\begin{equation*} \label{equ:qc_differential}
D\nu(x)(E_x) = B(0,r_x)
\end{equation*} 
for a.e. $x \in \Omega$ and some $r_x>0$. Let $C_x' = D\nu(x)(C_x)$, observing that $D\nu(x)$ differs from $T_x$ by a scaling factor and orthogonal transformation. 

Define $H = h \circ \nu^{-1}: \Omega \rightarrow X$. As above we obtain Lipschitz pieces $H_j = H|\Omega_j'$ for disjoint sets $\Omega_j' \subset \Omega$.  Then there exists $R_{x'}>0$ such that $C_{x}' = \{y \in \mathbb{R}^2: MD(H_j,x')(y) \leq R_{x'}\}$, for a.e. $x' \in \Omega$, where $j$ is such that $x' \in \Omega_j'$ and $x = \nu^{-1}(x')$. Hence for a.e. $x' \in \Omega$, the metric derivative satisfies $|MD(H_j,x')|= R_{x'}/r_x$, and the Jacobian $J_H$ is given by $J_H(x') = \pi R_{x'}^2/|C_{x}'|$. By Lemma \ref{lemm:convex_body} we see that
$$\frac{|MD(H_j,x')|^2}{J_H(x')} = \frac{|C_{x}'|}{\pi r_x^2} \leq \frac{4}{\pi} .$$
Following \cite{Raj16}, this suffices to show the inequality $\Mod \Gamma \leq 4\pi^{-1}\Mod H\Gamma$ for all curve families $\Gamma$ in $\Omega$.

Similarly we have
$$\ell(MD(H_j,x')) := \inf_{|z|=1} |MD(H_j,x')z| = \frac{R_{x'}}{L(C_{x}')}, $$
valid for a.e $x' \in \Omega$. This gives by Lemma \ref{lemm:convex_body}
$$\frac{J_H(x')}{\ell(MD(H_j,x'))^2} =\frac{\pi L(C_{x}')^2}{|C_{x}'|}  \leq \frac{\pi}{2}, $$
which suffices to show that $\Mod H\Gamma \leq (\pi/2)\Mod \Gamma $ for all curve families $\Gamma$ in $\Omega$. 
\end{proof}

{\bf \noindent Acknowledgments.} The author thanks Jeremy Tyson, Kai Rajala and Marius Junge for conversations related to the topic of this paper.

\bibliographystyle{abbrv}
\bibliography{biblio}
\end{document}